\documentclass[11pt,reqno]{BICA} 
\usepackage[utf8]{inputenc}
\usepackage{fullpage,amsfonts,amsmath,amsthm,amssymb}
\usepackage[justification=centering]{caption}
\usepackage{color,soul}
\usepackage{fullpage}
\usepackage{ytableau}
\usepackage{enumitem}   
\usepackage{hyperref,float}
\usepackage{mathtools}
\usepackage{xcolor}
\usepackage{thm-restate}
\usepackage{xurl}
\usepackage{thmtools}
\usepackage{hyperref}
\usepackage{cleveref}			
\usepackage{dsfont}
\usepackage{tikz}
\usepackage{quiver}
\usepackage{lineno}
\usepackage{bbold}
\usepackage{subcaption}
\usetikzlibrary{decorations.pathreplacing}
\theoremstyle{definition}
\newtheorem{notation}[theorem]{Notation}

\newcommand{\uu}{\textbf{u}}

\newcommand{\ch}{\mathrm{ch}}
\newcommand{\Ind}{\mathrm{Ind}}
\newcommand{\Sym}{\mathfrak{S}}
\newcommand{\NDPF}{\mathrm{PF}^{\uparrow}}
\newcommand{\NDDPF}{\mathrm{DPF}^\uparrow}
\newcommand{\PF}{\mathrm{PF}}
\newcommand{\DPF}{\mathrm{DPF}}

\newcommand{\sort}{\mathrm{sort}}
\newcommand{\Krew}{\mathrm{Krew}}
\newcommand{\runs}{\mathrm{runs}}
\newcommand{\dip}{\mathrm{dip}}
\newcommand{\Park}{{{\sf Park}}}
\newcommand{\DPark}{{{\sf DPark}}}
\newcommand{\x}{\mathbf{x}}
\newcommand{\y}{\mathbf{y}}
\newcommand{\z}{\mathbf{z}}
\newcommand{\w}{\mathbf{w}}
\newcommand{\de}{\mathrm{dft}}
\newcommand{\pde}{\mathrm{pdft}}
\newcommand{\NN}{\mathbb{N}}
\newcommand{\CC}{\mathbb{C}}

\newcommand{\bij}{\longleftrightarrow}
\newcommand{\SYT}{\mathrm{SYT}}
\newcommand{\pairset}{decrement pair set}
\newcommand{\setM}{fixed pair set}

\newcommand{\Cat}{{\mathcal{C}}}

\usepackage{hyperref}

\usepackage{lineno}

\begin{document}

\title{The Defective Parking Space and Defective Kreweras Numbers}

\author{%
Rebecca E. Garcia,
Pamela E. Harris,
Alex Moon,
Aaron Ortiz,
Lauren J. Quesada,
Cynthia Marie Rivera S\'anchez,
Dwight Anderson Williams II
and
Alexander N.~Wilson
}


\begin{abstract}
A 
defective $(m,n)$-parking function with defect $d$ is a parking function with $m$ cars attempting to park on a street with $n$ parking spots in which exactly $d$ cars fail to park. 
We establish a way to compute the defect of a defective $(m,n)$-parking function and show that the defect of a parking function is invariant under the action of $\mathfrak{S}_m$, the symmetric group on $[m]=\{1,2,\ldots,m\}$. 
We introduce the defective parking space $\DPark_{m,n}$ spanned by defective parking functions and describe its Frobenius characteristic as an $\mathfrak{S}_m$ representation graded by defect via coefficients $\Krew_{d,n}(\lambda)$ called defective Kreweras numbers. We provide a conjectured formula for $\Krew_{d,n}(\lambda)$ for sufficiently large $n$.
We also show that the set of nondecreasing defective $(m,n)$-parking functions with defect $d$ are in bijection with the set of standard Young tableaux of shape $(n + d, m - d)$.  
This implies that the number of $\mathfrak{S}_m$-orbits of defective $(m,n)$-parking functions with defect $d$ is given by 
$\frac{n-m+2d+1}{n+d+1}\binom{m+n}{n+d}$. 
We also give a multinomial formula for the size of an $\mathfrak{S}_m$-orbit of a nondecreasing $(m,n)$-parking function with defect $d$. We conclude by using these results to give a new formula for the number of defective parking functions.
\end{abstract}
\maketitle

\section{Introduction}

Given $m,n\in\mathbb{N}\coloneqq\{1,2,3,\ldots\}$ with $m\leq n$, we recall the classical combinatorial description of $(m,n)$-parking functions.
Consider $m$ cars in queue to enter a linear parking lot with $n$ sequentially numbered parking spaces. 
A tuple $\x=(x_1,x_2,\ldots,x_m)\in[n]^m$, where $[n]=\{1,2,\ldots,n\}$, is called a \emph{preference list} and it encodes that, for each $i\in[m]$, car $i$ prefers parking spot $x_i\in[n]$. Given a parking list $\x$, we refer to an entry as a \emph{parking preference}. The following procedure for parking is what we refer to as the \emph{classical parking scheme}.
Cars enter the street one at a time, proceed to their preferred parking spot, and park there if their preference is unoccupied. Otherwise, the car parks in the first space available after their preference. 
If no such spot is available, then we say that car has failed to park.
If the preference list $\x$ allows all cars to park, then $\x$ is called an $(m,n)$-parking function. 
Throughout we let $\PF_{m,n}$ denote the set of $(m,n)$-parking functions of length $m$, and we recall that when $n=m$, Konheim and Weiss showed $|\PF_{n,n}|=(n+1)^{n-1}$, while Cameron, Johannsen, Prellberg, and Schweitzer showed $|\PF_{m,n}|=(n+1-m)(n+1)^{m-1}$, see \cites{defective,Konheim_Weiss} for proofs.

Parking functions have connections to many areas of mathematics including representation theory, Coxeter groups, hyperplane arrangements, discrete geometry, the quicksort algorithm, the Tower of Hanoi game, cooperative game theory, and many others \cites{ Hanoi,unit_perm,Boolean_intervals,parking_games,haimanConjecturesQuotientRing1994,quicksort}.
Moreover, many have studied 
generalizations of parking functions allowing for cars to move backwards when finding their preferred parking spot occupied or bumping cars out of an occupied parking spot, cars parking on a street with some parking spots unavailable, allowing cars to have various lengths and colors, and cars having multiple parking spot preferences \cites{joinsplit,ell_interval_pfs,unit_pf,assortments1,backwards,knaples,sequences,vacillating,assortments2,MVP}.
There has also been substantial interest in the study of
statistics of parking functions. This includes statistics such as the number of lucky cars (those cars that park in their preference), displacement (the difference between where a car parks and its preference), and more classical statistics such as ascents, descents, ties, peaks and valleys, and flatness \cites{pf_stats,FlatPF,GesselSeo}. 

In our work, we study \textit{defective parking functions} as introduced by Cameron, Johannsen, Prellberg, and Schweitzer in  \cite{defective}. 
Classically, a defective parking function of defect $d$ is a preference list in $[n]^m$ in which exactly $d$ cars fail to park. 
We generalize this to think of cars parking on an infinitely long road, where defect $d$ is the difference between the true parking spot of the final car and $n$ whenever there are cars failing to park. Here a car ``fails to park" implies the car did not park in the first $n$ spots so that $d$ is nonnegative and only parking functions yield zero defect. 
We also allow preference lists to come from $[n+1]^m$ to include defective parking functions where no cars park in the $n$ parking spots.
In this context, a parking function is a preference list with zero defect. 
If $d \in \mathbb{Z}_{\geq0}$, we let $\DPF_{m,n,d}$ denote the set of all defective parking functions with defect $d$, and as usual $|\DPF_{m,n,d}|$ denotes the cardinality of the set.
We refer the reader to \cites{defective,butler2017parking} 
for the enumeration of defective parking functions.

Our first contribution provides a quick way to determine the defect of a  preference list.

\begin{definition}\label{def:dft}
    The defect of a parking preference $\x \in [n+1]^m$ is denoted $\de(\x)$.
\end{definition}

\begin{theorem}\label{thm:defect_characterization}
    For $m,n\in\mathbb{N}$, let $\x\in[n+1]^m$ and let $\x'=(x_1',x_2',\ldots,x_m')$ be the rearrangement of $\x$ into nondecreasing order. Then $\x$ is a defective parking function with defect $d$ if and only if \[\max_{j\in[m]}(x'_j-j+(m-n),0)=d.\]
\end{theorem} 

A common representation of parking functions is by labeled Dyck paths consisting of north and east steps: For a parking function $\y$ with weakly increasing rearrangement $\x$, the number of east steps before the $i$th north step is $x_i - 1$. Here $\y$ is determined by labeling the north steps such that columns are labeled in increasing order, and letting the corresponding permutation (read bottom-to-top) be the permutation of $\x$ that gives $\y$. For more on labeled Dyck paths, we refer the interested reader to \cite{adeniran-pudwell} and \cite{willigenburg-shuffle}. Theorem~\ref{thm:defect_characterization} allows for a similar characterization of defective parking functions using labeled lattice paths with a ``dip'' statistic, which we discuss more in Section~\ref{sec:rep theory}. The dip statistic corresponds to what we call predefect, which we introduce formally in \Cref{def:predefect}. 

\begin{example}
    Consider $\y = (9,3,5,6,10,9,5) \in \DPF_{7,9,2}$, which is a rearrangement of the defective parking function \[\x = (3,5,5,6,9,9,10) \in \NDDPF_{7,9,2}.\] 
    Note that $\y$ has defect $2$, by Theorem~\ref{thm:defect_characterization}, and $\x$ and $\y$ can be represented as labeled lattice paths as illustrated in Figure~\ref{fig:labeled-paths}.
    \begin{figure}[h!]\begin{subfigure}{0.5\linewidth}\begin{tikzpicture}[scale=0.5]
        \draw[step=1,gray,thin] (0,0) grid (9,7);
        \draw[red, thick] (2,0) -- (9,7);
        \draw[blue,ultra thick] (0,0) -- (2,0) -- (2,1) -- (4,1) -- (4,3) -- (5,3) -- (5,4) -- (8,4) -- (8,6) -- (9,6) -- (9,7);
        \node at (2.5, 0.5) {1};
        \node at (4.5, 1.5) {2};
        \node at (4.5, 2.5) {3};
        \node at (5.5, 3.5) {4};
        \node at (8.5, 4.5) {5};
        \node at (8.5, 5.5) {6};
        \node at (9.5, 6.5) {7};
    \end{tikzpicture}\caption{The labeled lattice path representation of $\x = (3,5,5,6,9,9,10) \in \NDDPF_{7,9,2}$.}\end{subfigure}
    \begin{subfigure}{0.5\linewidth}\begin{tikzpicture}[scale=0.5]
        \draw[step=1,gray,thin] (0,0) grid (9,7);
        \draw[red, thick] (2,0) -- (9,7);
        \draw[blue,ultra thick] (0,0) -- (2,0) -- (2,1) -- (4,1) -- (4,3) -- (5,3) -- (5,4) -- (8,4) -- (8,6) -- (9,6) -- (9,7);
        \node at (2.5, 0.5) {2};
        \node at (4.5, 1.5) {3};
        \node at (4.5, 2.5) {7};
        \node at (5.5, 3.5) {4};
        \node at (8.5, 4.5) {1};
        \node at (8.5, 5.5) {6};
        \node at (9.5, 6.5) {5};
    \end{tikzpicture}\caption{The labeled lattice path representation of $\y = (9,3,5,6,10,9,5) \in \DPF_{7,9,2}$.}\end{subfigure}\caption{Two labeled lattice paths. Defect is given by dip, the maximum distance of the lattice path below the diagonal red line.}\label{fig:labeled-paths} \end{figure}
\end{example}

\begin{remark}
    In the study of parking functions, most authors add the restriction that there not be more cars than spots --- indeed if $m > n$ then some amount of cars fail to park. In our case, however, we have no such restriction. When there are more cars than spots, we simply observe that defect $d$ is positive and no less than $m - n$. 
\end{remark}

We recall that an important feature of parking functions is that they are invariant under the action of $\mathfrak{S}_m$, the symmetric group on $[m]$, see \cite[\S 4 Example 1]{Francon}. 
Namely, given $\alpha=(a_1,a_2,\ldots,a_m)\in\PF_{m,n}$ and any $\pi\in\mathfrak{S}_m$, the tuple
\[\pi\cdot\alpha=(a_{\pi^{-1}(1)},a_{\pi^{-1}(2)},\ldots,a_{\pi^{-1}(m)})\in\PF_{m,n}.\]
It is a well-documented result (see \cite{ArmstrongCatalan} for a proof) that the $\mathfrak{S}_n$-orbits of $\PF_{n,n}$
are enumerated by the Catalan numbers, which are defined by \cite[\href{https://oeis.org/A000108}{A000108}]{OEIS}: \[\Cat_n=\frac{1}{n+1}\binom{2n}{n}.\]
There are multiple ways to consider parking functions with unequal number of cars and spots. Armstrong, Loehr, and Warrington~\cite{armstrong} consider the case where 
$m$ and $n$ are coprime and in which case a car takes up a fraction of a spot.
In this paper, we allow for $m$ and $n$ to be any positive integers and for each car to occupy exactly one spot. 
Our first result establishes that the defect of a parking function is invariant under the action of $\mathfrak{S}_m$, although this fact is already implicit in \cite{defective}. 
This invariance of defect allows us to consider $\DPark_{m,n}$, the permutation representation of $\mathfrak{S}_m$ acting on defective parking functions graded by defect. 
We compute the graded Frobenius characteristic for $\DPark_{m,n}$ in terms of \emph{defective Kreweras numbers} $\Krew_{d,n}(\lambda)$, which we formally introduce in \Cref{sec:rep theory}.

\begin{theorem}
    The graded Frobenius characteristic $\ch_t(\DPark_{m,n})$ is given by \begin{align*}
    \ch_t(\DPark_{m,n})&\hspace{-1mm}=\hspace{-1mm}\sum_{\lambda\vdash m}\hspace{-1mm}\left(\sum_{i=0}^{m-n}\hspace{-1mm}\Krew_{i,n}(\lambda)+\hspace{-3mm}\sum_{d>m-n}\hspace{-3mm}\Krew_{d-(m-n),n}(\lambda)t^d\hspace{-1mm}\right)\hspace{-1mm}h_\lambda,
    \end{align*}
    where $h_{\lambda}$ denotes a complete homogeneous symmetric function.
    In particular, when $n=m$, \begin{align*}
        \ch_t(\DPark_{n,n})&=\sum_{\lambda\vdash n}\sum_{d\geq 0}\Krew_{d,n}(\lambda)t^dh_\lambda.
    \end{align*}
\end{theorem}

We then establish the following bijection.

\begin{theorem}\label{thm:general}
The set of nondecreasing defective parking functions of defect $d \in \mathbb{Z}_{\geq0}$, denoted $\DPF_{m,n,d}^{\uparrow},$ is in bijection with the set of standard Young tableaux of shape $(n+d,m-d)$.
\end{theorem}

As a consequence of Theorem~\ref{thm:general} and the hook length formula (see \cite{FRT} or \cite[Section 3.10]{Sagan}), we establish a formula for the number of $\mathfrak{S}_m$-orbits of the set of defective parking functions with defect~$d$.

\begin{corollary}\label{cor:counting}
If $m,n\in\mathbb{N}$ and $d \in \mathbb{Z}_{\geq0}$, then
    \begin{align}\label{eq:cat convolution}
    |\NDDPF_{m,n,d}| &= \frac{n-m+2d+1}{n+d+1}\binom{m+n}{n+d}.
    \end{align}
\end{corollary}
When $d=0$ all cars can park and Equation~\eqref{eq:cat convolution} reduces to the Catalan numbers when $n=m$. 
Moreover, in Theorem~\ref{thm:size of orbit}, we provide a multinomial formula for the size of the $\Sym_{m}$-orbit given a nondecreasing defective parking function. 
We then use this result to give the following  new formula for the number of defective parking functions as multiset permutations of nondecreasing defective parking functions. 
\begin{theorem}\label{thm:new formula}
If $m,n,d\in\NN$,
then
    \begin{align}\label{eq:full count}
    |\DPF_{m,n,d}| &= \sum_{\x \in \NDDPF_{m,n,d}}\binom{m}{\ell_1(\x),\ell_2(\x),\ldots,\ell_{n}(\x)},
    \end{align}
   where, for $1\leq i\leq n$, $\ell_i(\x)$ denotes the number of elements of $\x$ whose value is $i$.
\end{theorem}

The paper is organized as follows.
In Section~\ref{sec:background}, we give a way to determine the defect of a preference list and use this result to prove Theorem~\ref{thm:defect_characterization}.
In Section~\ref{sec:rep theory},
we give a formula for the graded Frobenius characteristic of the space of defective parking functions.
In Section~\ref{sec:main bijection}, we prove Theorem~\ref{thm:general}
by establishing a set of bijections from which the result follows.
Section~\ref{sec:enumerative} contains proofs of the enumerative results in Corollary~\ref{cor:counting}, Theorem~\ref{thm:size of orbit}, and Theorem~\ref{thm:new formula}. We conclude with directions for future research.

\section{A Formula for Defect and \texorpdfstring{$\mathfrak{S}_m$}{Sn}-Invariance}\label{sec:background}
Prominent literature defines parking functions as elements of $[n]^m$, see \cite{Konheim_Weiss} or \cite[Exercise 5.29.\ a.]{StanleyECVol2}. As we are concerned with defective parking functions where not all cars may manage to park, we include the possible preference $n+1$ to indicate the possibility of a car preferring none of the available spots. Considering parking functions as elements of $[n+1]^m$ facilitates clarity when we introduce the conjugate of a parking function (Definition~\ref{def:conjugate-lat}). Moreover, it creates the possibility of a defective parking function where no cars manage to park, in which all cars simply enter and exit the parking lot without desiring to park. 
A  preference list that includes $n+1$ is not a parking function, so in the study of standard parking functions this difference in convention is immaterial.
Thus, we employ the following definitions and notation.

\begin{notation} \label{notation}Throughout the article we use the following notation and definitions.
\begin{itemize}[leftmargin=.2in]
\item We use bold font to denote that $\x$ is a  preference list and we let $x_i$ denotes the $i$th element of $\x$.
\item Given a set of preference lists $S$ and a condition $P$ on $S$, we let $S(P)$ denote the subset of elements of $S$ that satisfy $P$.
\item Throughout a \textit{preference list} is an element $\x \in[n+1]^m$.
\item A \textit{parking function} is a preference list in which no car parks beyond the $n$th spot using the classical parking scheme.
\item The set of parking functions with $m$ cars and $n$ spots is denoted $\PF_{m,n}$.
\item  The subset  of $\PF_{m,n}$ consisting of nondecreasing parking functions is denoted by $\NDPF_{m,n}$. 
\end{itemize}
\end{notation}


\begin{definition}\label{def:nonneg defect}
    A \textit{defective parking function} with defect $d$ is a  preference list in which exactly $d$ cars fail to park under the classical parking scheme. 
    The set of defective parking functions with $m$ cars, $n$ spots, and defect $d$ is denoted by $\DPF_{m, n,d}$. The subset of $\DPF_{m,n,d}$ whose elements are nondecreasing is denoted by $\NDDPF_{m,n,d}$.
\end{definition}

\begin{example}\label{ex:defective}
 Let $m=7$ and $n=9$. Consider the nondecreasing preference list $(3,5,5,6,9,9,10)\in[10]^7$. Cars park/fail to park as illustrated in Figure~\ref{fig:example1}. Since the preference list was nondecreasing, the cars parked (on an infinitely long parking lot) in the order they arrived. 
The two cars that park in a spot numbered $j \geq n+1 = 10$ failed to park in the parking lot.  Hence, the defect of $(3,5,5,6,9,9,10)$ is two. 
\begin{figure}[!ht]
    \centering
    \resizebox{\textwidth}{!}{
    \begin{tikzpicture}
    \draw[step=1cm,gray,very thin] (0,0) grid (11,1);
    \draw[fill=gray!50] (2.1,0.1) rectangle (2.9,.9);
     \node at (2.5,.5) {$1$};
     \draw[fill=gray!50] (4.1,0.1) rectangle (4.9,.9);
     \node at (4.5,.5) {$2$};
     \draw[fill=gray!50] (5.1,0.1) rectangle (5.9,.9);
     \node at (5.5,.5) {$3$};
     \draw[fill=gray!50] (6.1,0.1) rectangle (6.9,.9);
     \node at (6.5,.5) {$4$};
     \draw[fill=gray!50] (8.1,0.1) rectangle (8.9,.9);
    \node at (8.5,.5) {$5$};
     \draw[fill=red!50] (9.1,0.1) rectangle (9.9,.9);
    \node at (9.5,.5) {$\textbf{6}$};
     \draw[fill=red!50] (10.1,0.1) rectangle (10.9,.9);
    \node at (10.5,.5) {$\textbf{7}$};
    \node at (11.05,-.25){cars that failed to park};
     \node at (.5,-.25) {$1$};
    \node at (1.5,-.25) {$2$};
    \node at (2.5,-.25) {$3$};
    \node at (3.5,-.25) {$4$};
    \node at (4.5,-.25) {$5$};
    \node at (5.5,-.25) {$6$};
    \node at (6.5,-.25) {$7$};
    \node at (7.5,-.25) {$8$};
    \node at (8.5,-.25) {$9$};
    \draw[ultra thick](9,-.5)--(9,1.5);
    \node at (9,1.75){end of lot};
    \end{tikzpicture}
    }
    \caption{Parking position of cars with preference list $(3,5,5,6,9,9,10)\in\NDDPF_{7,9,2}$.
}
    \label{fig:example1}
\end{figure}
\end{example}

\begin{definition}\label{def:predefect}
    For $\x\in[n+1]^m$, the \emph{predefect sequence} of $\x$ is 
    \[\gamma(\x) = (x_1 - 1, x_2 - 2, \ldots, x_m - m).\]
    For $i \in [m]$, the $i$th entry $\gamma(\x)_i$ of the predefect sequence $\gamma(\x)$ is the \emph{predefect of $\x$ at $i$}. The \emph{predefect} of $\x$ is 
    \[\pde(\x)=\max(\gamma(\x)).\]
\end{definition}

When we compute the defect of a parking function with $m$ cars and $n$ spots, we find it necessary to normalize by $m-n$ and ensure nonnegativity. Therefore, we define analogous terms to predefect for defect.

\begin{definition} 
If $\x \in \NDDPF_{m,n,d}$, the \emph{defect sequence} of $\x$ is \[\delta(\x)=(\gamma(\x)_1 + (m-n), \gamma(\x)_2 + (m-n), \ldots, \gamma(\x)_m + (m-n)).\]
For $i\in [m]$, the  $i$th entry $\delta(\x)_i$ of the defect sequence $\delta(\x)$ is called the \textit{defect of $\x$ at $i$.} 
\end{definition}
In Theorem~\ref{thm:defect_characterization}, we prove that the defect of $\x$ is
\[\de(\x) = \max_{j \in [m]}(\delta(\x)_j, 0) = \max(\pde(\x) + (m-n), 0).\]

\begin{remark}In the case that $m=n$ and $d=0$, $f: \x \mapsto -\delta(\x)$ gives a bijection between $\NDPF_{n,n}$ and Catalan words. For more on Catalan words we refer the interested reader to \cite[Exercise 6.19.\ u, p.\ 222]{StanleyECVol2}.\end{remark}

\begin{example}[Example~\ref{ex:defective} continued]\label{ex:nonneg defect}
Consider the tuple $\x=(3,5,5,6,9,9,10)$ in $\NDDPF_{7,9,2}$. Note 
$\delta(\x)=
(0,1,0,0,2,1,1)$.
The maximum value in $\delta(\x)$ is two, which happens to be $\de(\x)$, the defect of $\x$. 
\end{example}

\begin{lemma}
If $\x\in\NDDPF_{m,n,d}$, then $\max(m-n, 0)\leq \de(\x)\leq m$. 
\end{lemma}

\begin{proof}
    There are $n-m$ more parking spots than cars. The minimum possible defect for any preference list in $[n+1]^m$ is $\max(m-n,0)$, which is achieved whenever all of the cars park in the first $m$ parking spots. 
    
The maximum possible defect for any preference list in $[n+1]^m$ is $m$, which is achieved whenever all of the cars prefer the $(n+1)$st parking spot and all fail to park in the first $n$ parking spots.
Hence $\max(m-n,0)\leq \de(\x)\leq m$.
\end{proof}

Next we prove that in fact the defect of a nondecreasing parking list $\x$ is always the maximum entry in $\delta(\x)$. This allows us to determine the defect of a  preference list without having to employ the parking scheme. 

\begin{lemma}\label{prop:defect2}
If $\x$ is a nondecreasing element of $[n+1]^m$, then 
\[\de(\x)=\max(\delta(\x),0).\]
\end{lemma}

\begin{proof}
    Consider first the case of $m$ cars parking in an infinite parking lot with a nondecreasing preference list. We call a spot \emph{wasted} if there is no car in that spot and there is a car in some spot to its right. If we write $w_k$ for the number of wasted spots immediately after car $k$ parks, then we have that the parking position of car $k$ satisfies $p_k=k+w_k$.

    Because the preference list is nondecreasing, the number of new wasted spots created when car $k+1$ parks is given by the number of spots after $p_k$ and before $x_{k+1}$. That is, \begin{align*}
        w_{k+1}-w_k&=\max\{x_{k+1}-p_k-1,0\}\\
        &=\max\{x_{k+1}-(k+1)-w_k,0\}.\\
        \intertext{Then, adding $w_k$ to both sides yields}
        w_{k+1}&=\max\{x_{k+1}-(k+1),w_k\}.
    \end{align*}

    Because $w_1=x_1-1$ and each $w_k$ is obtained recursively as above, we have that \begin{align*}
        w_k&=\max_{j\in[k]}\{x_j-j\}
    \end{align*} for $1\leq k\leq m$.

    Note that if no car's preference exceeds $n+1$, then there can be no wasted spaces past spot $n$. In this case, the number of cars that park in a position larger than $n$ is computed as $n$ less than the parking position of car $m$ (or zero if this subtraction is negative). That is, \begin{align*}
        \de(\x)&=\max\{p_m-n,0\}\\
        &=\max\{m+w_m-n,0\}\\
        &=\max_{j\in[m]}\{x_j-j+(m-n),0\}\\
        &=\max_{j\in[m]}(\delta(\x)_j,0).\qedhere
    \end{align*}
\end{proof}

We conclude this section with a proof that the defect is invariant under the action of the symmetric group.

\begin{proof}[[Proof of Theorem~\ref{thm:defect_characterization}]]
 Stanley and Pitman \cite{PitmanStanley} introduced the following definition: For an integer tuple $\textbf{u}=(u_1,u_2,\ldots,u_k)$ with $k\in\NN$, the set of $\textbf{u}$-parking functions consist of the sequences
$(a_1,a_2,\ldots,a_k)$ whose rearrangements into nondecreasing order, denoted by $(b_1,b_2,\ldots,b_k)$, satisfy $b_i \leq u_1+\cdots+u_i$ for all $1\leq i\leq k$. 
In this way, when $\uu=(1,1,\ldots,1)\in[n]^m$, the set of $\uu$-parking functions is precisely the set of $(m,n)$-parking functions. Moreover, Cameron et.~al~\cite[p.\ 3]{defective} 
give a characterization of defective parking functions with $m$ cars, $n$ parking spots, and defect $d$ as \[(n-(m-d)+1\;,\;\underbrace{1,\;1,\;\ldots\;,1}_{\begin{matrix}2m - n - d\\ \mbox{copies}\end{matrix}}\;,\;\underbrace{0,\;0,\;\ldots\;,0}_{\begin{matrix}d+(n-m)\\ \mbox{copies}\end{matrix}})\mbox{-parking functions}\] that are not \[(n-(m-d)\;,\;\underbrace{1,\;1,\;\ldots\;,1}_{\begin{matrix}2m - n - d\\ \mbox{copies}\end{matrix}}\;,\;\underbrace{0,\;0,\;\ldots\;,0}_{\begin{matrix}d+(n-m)\\\mbox{copies}\end{matrix}})\mbox{-parking functions.}\] As this characterization holds for any initial parking function with defect $d$, the defect only depends on the content of the parking function and not on the order of the entries. Combining this fact with \Cref{prop:defect2}, we obtain the desired result.
\end{proof}

\section{Defective Parking Spaces}\label{sec:rep theory}

In this section, we consider permutation representations of the symmetric group arising from defective parking functions. We assume familiarity with the representation theory of the symmetric group and symmetric function terminology (see \cite{Sagan} for a reference on these topics).

For a set $S$, let $\CC S$ denote the vector space of formal linear combinations of elements of $S$. Let $\Park_{m,n}=\CC\PF_{m,n}$. The \emph{classical parking space} $\Park_{n,n}$ appears in formulas for the decomposition of the $\mathfrak{S}_m$-representation $D=R/(R_+^{\mathfrak{S}_m})$ where \[R=\CC[x_1,x_2,\ldots,x_m,y_1,y_2,\ldots,y_m]\] and $R_+^{\mathfrak{S}_m}\subseteq R$ is the ideal generated by elements of $R$ without a constant term that are invariant under simultaneous permutation of the variables $x_i$ and $y_i$. This space is called the \emph{diagonal coinvariants} (see \cite{haiman2001vanishing}).

The \emph{Frobenius characteristic map} is an isomorphism between the ring $R$ of characters for symmetric group representations and the ring $\Lambda$ of symmetric functions. Identifying a representation $V$ of $\mathfrak{S}_m$ with its character function $\chi_V:\mathfrak{S}_m\to\CC$, we write \begin{align*}
    \ch(V)&=\frac{1}{m!}\sum_{\sigma\in \mathfrak{S}_m}\chi_V(\sigma)p_{c(\sigma)}
\end{align*} where $c(\sigma)$ is the integer partition denoting the cycle type of $\sigma$ and $p_\lambda$ are the \emph{power sum symmetric functions}. For our purposes, the basis of $\Lambda$ we are most interested in is that of the \emph{complete homogeneous symmetric functions} $h_\lambda$, which are the image of the induction of the trivial representation $\mathbb{1}$ of the Young subgroup $\mathfrak{S}_\lambda=\mathfrak{S}_{\lambda_1}\times \mathfrak{S}_{\lambda_2}\times\cdots\times \mathfrak{S}_{\lambda_k}$ to $\mathfrak{S}_m$. That is, \begin{align*}
    h_\lambda=\ch\left(\Ind_{\mathfrak{S}_\lambda}^{\mathfrak{S}_m}(\mathbb{1})\right)
\end{align*} for $\lambda\vdash m$.

For $\lambda\vdash m$ of length $k$ let \begin{align}
    \Krew(\lambda)=\frac{1}{m+1}{m+1\choose m+1-k,\mu_1(\lambda),\mu_2(\lambda),\ldots,\mu_m(\lambda)}\label{eq:kreweras_number}
\end{align} be the \emph{Kreweras number} of $\lambda$
introduced by Kreweras in \cite{K72}. It is well-known (see \cite{haiman1994conjectures} or \cite[Proposition 2.2]{stanley1997parking}) that \begin{align*}
    \ch(\Park_{n,n})=\sum_{\lambda\vdash n} \Krew(\lambda)h_\lambda.
\end{align*} 

\begin{definition}
    Let $\DPark_{m,n}=\CC [n+1]^m$. This vector space has a grading by defect. That is, we let \[\DPark_{m,n}^{(d)}=\CC\{\x:\de(\x)=d\},\] and we call $\DPark_{m,n}$ with this grading the \emph{defective parking space}.
\end{definition}

 By \Cref{thm:defect_characterization} each subspace $\DPark_{m,n}^{(d)}$ is an $\mathfrak{S}_m$-representation, so we can compute a graded Frobenius characteristic \[\ch_t(\DPark_{m,n})=\sum_{d=\max(m-n,0)}^{m}\ch(\DPark_{m,n}^{(d)})t^d.\] 

\begin{remark}
    This graded Frobenius characteristic $\ch_t(\DPark_{m,n})$ has two interesting specializations:
    \begin{enumerate}
        \item When $t=0$, we have that $\ch_0(\DPark_{m,n})=\ch(\Park_{m,n})$ where $\mathfrak{S}_m$ acts by permuting the order of the preference list.
        \item When $t=1$, we have that $\ch_1(\DPark_{m,n})=\ch(\left(\CC^{n+1}\right)^{\otimes m})$ where $\mathfrak{S}_m$ acts by permuting tensor factors.
    \end{enumerate}
\end{remark}

\begin{lemma}\label{lem:fixed_content_frob}
    For $\alpha=(a_1,a_2,\ldots,a_{n+1})$ a weak composition of $m$, let $V_\alpha\subseteq \DPark_{m,n}$ be the span of words $w$ where the multiplicity of $i$ in $w$ is $\alpha_i$. Then \begin{align*}
        \ch(V_\alpha)=h_\alpha.
    \end{align*}
\end{lemma}

\begin{proof}
    The action of $\mathfrak{S}_m$ on $V_\alpha$ is isomorphic to the permutation representation of $\mathfrak{S}_m$ acting by left-multiplication on the right cosets of the Young subgroup $\mathfrak{S}_\alpha\subseteq \mathfrak{S}_m$. Hence, $V_\alpha$ is isomorphic to the induction of the trivial representation of $\mathfrak{S}_\alpha$ to $\mathfrak{S}_m$.
\end{proof}

\begin{definition}\label{def:defkreweras}
    Given $d\geq0$, define the \emph{defective Kreweras number} $\Krew_{d,n}(\lambda)$ of a partition $\lambda\vdash m$ of length $k$ to be the number of tuples $\x\in[n+1]^m$ of the form $\x=({a_1}^{\lambda_{\sigma(1)}},{a_2}^{\lambda_{\sigma(2)}},\ldots,{a_k}^{\lambda_{\sigma(k)}})$ where ${a_i}^{\lambda_{\sigma(i)}}$ represents a string of $\lambda_{\sigma(i)}$ copies of $a_i$ for some $0<a_1<a_2<\cdots<a_k\leq n+1$ and $\sigma\in \mathfrak{S}_k$ such that $\max(\pde(\x),0)=d$.
\end{definition}

\begin{remark}
    When $d=0$ and $m=n$, the defective Kreweras numbers specialize to the ordinary Kreweras numbers, enumerating nondecreasing parking functions of the form given in \Cref{def:defkreweras}.
\end{remark}

\begin{example} Let $\lambda=(2,1)$ and $n=3$. In the following table we provide some computations:
    \[\begin{array}{c|c||c|c}
        \x & \max(x_i-i,0) & \x & \max(x_i-i,0)\\\hline
        (1,1,2) & 0 & (1,2,2) & 0 \\
        (1,1,3) & 0 & (1,3,3) & 1 \\
        (1,1,4) & 1 & (1,4,4) & 2 \\
        (2,2,3) & 1 & (2,3,3) & 1 \\
        (2,2,4) & 1 & (2,4,4) & 2 \\
        (3,3,4) & 2 & (3,4,4) & 2
    \end{array}\]

    So, \[
        \Krew_{0,3}(2,1)=3,\quad
        \Krew_{1,3}(2,1)=5,\mbox{ and}\quad
        \Krew_{2,3}(2,1)=4.
    \]
\end{example}

\begin{theorem}
    The graded Frobenius characteristic $\ch_t(\DPark_{n,n})$ is given by
    \begin{align*}
        \ch_t(\DPark_{n,n})&=\sum_{\lambda\vdash n}\sum_{d\geq 0}\Krew_{d,n}(\lambda)t^dh_\lambda.
    \end{align*}
    More generally, the graded Frobenius characteristic $\ch_t(\DPark_{m,n})$ is given by \begin{align*}
        \ch_t(\DPark_{m,n})&=\\
        &\hspace{-.5in}\sum_{\lambda\vdash m}\left(\sum_{d=0}^{n-m}\Krew_{d,n}(\lambda)+\sum_{d>m-n}\Krew_{d-(m-n),n}(\lambda)t^d\right)h_\lambda.
    \end{align*}
\end{theorem}

\begin{proof} For $\alpha$ a weak composition of $m$, let \[\x(\alpha)=(1^{\alpha_1},2^{\alpha_2},\ldots).\] Moreover, let $\alpha \vDash_{j}k$ denote a weak composition of $k$ into $j$ parts. By \Cref{lem:fixed_content_frob}, we compute the graded Frobenius characteristic $\ch_t(\DPark_{m,n})$ as

    \begin{align*}
    \ch_t(\DPark_{m,n})&=\sum_{\alpha\,\vDash_{n+1}\,m}t^{\de(\x(\alpha))}h_\alpha\\
&=\sum_{\lambda\,\vdash\,m}
a_{\lambda,n}(t)h_\lambda\\
\intertext{where}
        a_{\lambda,n}(t)&=\sum_{
    \substack{\alpha\,\vDash_{n+1}\,m
\\\sort(\alpha)=\lambda}}t^{\de(\x(\alpha))}.
    \end{align*}
    
    Note that $a_{\lambda,n}(t)=\sum_{d\geq0}c_dt^d$ where for $d\geq 0$, $c_d$ is the number of nondecreasing tuples $\x\in[n+1]^m$ with multiplicities of entries given by $\lambda$ and $\de(\x)=d$. As $\de(\x)=\max_{j\in[m]}(x_j-j+(m-n),0)$, each tuple $\x$ with $\max_{j\in[m]}(x_j-j)\leq m-n$ contributes to $c_0$, hence \[c_0=\sum
    _{d=0}^{m-n}\Krew_{d,n}(\lambda).\]
    For $d>m-n$, we have that $c_d=\Krew_{d-(m-n),n}(\lambda)$.
\end{proof}

We conclude this section with results on interpreting and computing the defective Kreweras numbers $\Krew_{d,n}(\lambda)$. We recall that an $(n, m)$-lattice path is a path from $(0, 0)$ to $(n,m)$ consisting of $(0, 1)$ steps (called north steps and denoted by $N$) and $(1, 0)$ steps (called east steps and denoted by $E$). Such a path is typically denoted by a word with $m$ $N$'s and $n$ $E$'s. Given a lattice path $\w$, we let $n_i$ denote the position of the $i$th $N$ step in $\w$. The  preference list $\x$ corresponding to $\w$ is defined by
\begin{align}
    x_i = 1 + \left|\{j \in [m+n] \mid w_j = E, j < n_i\}\right|. \label{eq:pref_path_bij}
\end{align}
In other words, for each $i\in[m]$, the entry $x_i$ is one more than the number of east steps before the $i$th north step. For an $(n,m)$-lattice path $\w$, let $\dip(\w)$ be the vertical distance from the line $y=x$ to the lowest point of $\w$ on or below the line $y=x$ (where $\dip(\w)=0$ if no such point exists). Let $\runs(\w)\vdash m$ be the integer partition consisting of the sizes of runs of the north steps of $\w$.

\begin{example} In Figure~\ref{fig:picture} we illustrate the $(7,5)$-lattice path $\w=ENEENNENENEE$ corresponding to $(2, 4, 4, 5, 6)$ as well as its $\dip$ and $\runs$ statistics.
\begin{figure}[h]
    \centering
\begin{tikzpicture}
\draw[gray] (0,0) grid (7,5);
\draw[red, ultra thick](0,0)--(1,0)--(1,1)--(3,1)--(3,3)--(4,3)--(4,4)--(5,4)--(5,5)--(7,5);
\draw[black,dashed](0,0)--(5,5);
\end{tikzpicture}
    \caption{A $(7,5)$-lattice path with $
    \dip(\w)=2$ and 
    $\runs(\w)=(2,1,1,1)$.
}
    \label{fig:picture}
\end{figure}
\end{example}

\begin{proposition}
    Let $\lambda\vdash m$. The defective Kreweras number $\Krew_{d,n}(\lambda)$ is the number of $(n,m)$-lattice paths with $\dip(\w)=d$ and $\runs(\w)=\lambda$.
\end{proposition}

\begin{proof}
    Given an $(n,m)$-lattice path $\w$, let $\x$ be its corresponding nondecreasing preference list $\x\in[n+1]^m$ given by Equation~\ref{eq:pref_path_bij}. Note that each north step in a run of north steps has the same number of east steps before it. Hence, \[\mathrm{runs}(\w)=\sort(\mu_1(\x),\mu_2(\x),\ldots,\mu_{n+1}(\x))\] where $\mu_i(\x)$ is the number of times the number $i$ appears in the tuple $\x$, hence $\x$ is of the form given in \Cref{def:defkreweras}. 

    Note also that just before the $i$th north step, there have been $x_i-1$ east steps and $i-1$ north steps. If the point at the bottom of this north step is below the line $y=x$, its horizontal distance (and hence the vertical distance) to the line $y=x$ is given by $x_i-1-(i-1)=x_i-i$. Hence, $\dip(\w)=\max_{j\in [m]}(x_j-j,0)$.
\end{proof}

Next, we provide sufficient conditions for $\Krew_{d,n}(\lambda)$ to be zero.

\begin{proposition}
    Let $\lambda\vdash m$ have length $k$. \begin{enumerate}
    \item For $n<k-1$, we have that $\Krew_{d,n}(\lambda)=0$ for all $d\geq0$.
    \item For $n\geq k-1$, we have that $\Krew_{d,n}(\lambda)=0$ for all $d> m-k+1$.
    \end{enumerate}
\end{proposition}

\begin{proof}
    If $n<k-1$, then there can be no weak composition of length $n+1$ with $\sort(\alpha)=\lambda$. So $\Krew_{d,n}(\lambda)=0$ for all $d\geq 0$.

    Suppose instead that $n\geq k-1$. If $\sort(\alpha)=\lambda$, then $\x=\x(\alpha)$ must be of the form \[\x=({a_1}^{\lambda_{\sigma(1)}},{a_2}^{\lambda_{\sigma(2)}},\ldots,{a_k}^{\lambda_{\sigma(k)}})\] where $a_1<a_2<\cdots<a_k$ and $\sigma\in \mathfrak{S}_k$. In the run ${a_i}^{\lambda_{\sigma(i)}}$ of identical values, the largest contribution to $\delta$ is at the first of these entries. That is, if we set 
    \[b_i=a_i-\left(\sum_{k=1}^{i-1}\lambda_{\sigma(k)}\right)-1+(m-n)\] to be the value of $\delta(\x)$ at the first index of the $i$th run of identical values, we have that $\de(\x(\alpha))=\max(b_1,b_2,\ldots,b_k,0)$.

    Suppose that $\x$ has the maximum possible defect among tuples of this form. Note that increasing a value $a_i$ only increases the corresponding $b_i$, so we may assume that $a_1=n-k+2$, $a_2=n-k+3$, \ldots, $a_k=n+1$ are their highest possible values. Then, we compute 
    \begin{align*}
        b_{i+1}-b_i&=a_{i+1}-a_i-\lambda_{\sigma(i)}\\
        &<a_{i+1}-a_i\\
        &=1.
    \end{align*}
    Hence, $b_{i+1}\leq b_i$, meaning \[\de(\x)=\max(b_1,b_2,\ldots,b_k,0)=\max(b_1,0)=\max(m-k+1,0).\]

    Because $\x$ was assumed to have the maximum possible defect among tuples of this form, $\Krew_{d,n}(\lambda)=0$ for $d>m-k+1$.
\end{proof}

Finally, we have a conjectured formula for $\Krew_{d,n}(\lambda)$ when $n$ is sufficiently large that closely mirrors the formula for $\Krew(\lambda)$ (c.f.\ Equation~\ref{eq:kreweras_number}).

\begin{conjecture}\label{conj:defective Kreweras}
    For $n\geq d+m-1$ and $\lambda\vdash m$ of length $k$,
    \begin{align*}
        \Krew_{d,n}(\lambda)&=\\
        &\hspace{-.5in}\frac{m+dk}{m+d}\frac{1}{m+d+1}{m+d+1\choose m+d+1-k,\mu_1(\lambda),\mu_2(\lambda),\ldots,\mu_m(\lambda)}
    \end{align*} 
    where $\mu_i(\lambda)$ is the number of parts of $\lambda$ of size $i$.
\end{conjecture}

\section{Main Bijection}\label{sec:main bijection}
In this section, we prove Theorem~\ref{thm:general}  and Corollary~\ref{cor:counting} through a chain of bijective arguments. 
To begin we consider the case where there are the same number of cars as numbered parking spots on the street.

\subsection{Equal Number of Cars and Spots}
Throughout this section we let $m=n$, in which case $d\geq m-n=0$.
Before proceeding, we point the reader to \Cref{fig:inductive-diagram}, which provides a visual of the bijections we define and the inductive argument we undergo in this section. 

Once we establish the bijections $\varphi$ and $\psi$, we show in the proof of \Cref{thm:main-bij} that given a bijection $\rho_i$ satisfying certain properties, we can lift it to a bijection $\rho_{i+1}$ satisfying the same properties. Repeating this process until $\rho_d$, we obtain a bijection between the set of nondecreasing defective $(m,n)$-parking functions with defect $d$ whose last entry is 
less than or equal to $k\in[n]$, which we denote  by 
$\NDDPF_{n,n,d}(x_n\leq k)$, 
and the set of nondecreasing parking functions of length $n+d$ whose last entry is less than or equal to $k-d$, which we denote by $\NDPF_{n+d,n+d}(x_{n+d}\leq k-d)$.

\begin{figure}[h!]
    \centering
\[\begin{tikzcd}
	{\NDDPF_{n,n,i+1}(x_n \leq k-d+i+1)} && {\NDPF_{n+i+1,n+i+1}(x_{n+i+1} \leq k-d)} \\
	\\
	{N_{n,i}(x_n\leq k-d+i)} && {M_{n+i}(x_{n+i}\leq k-d)} \\
	\\
	{\NDDPF_{n,n,i}(x_n \leq k-d+i)} && {\NDPF_{n+i,n+i}(x_{n+i} \leq k-d)}
	\arrow["{\rho_{i+1}}", dashed, from=1-1, to=1-3]
	\arrow["\varphi"', tail reversed, from=1-1, to=3-1]
	\arrow["{\text{(\Cref{lem:dpf-bij})}}", draw=none, from=1-1, to=3-1]
	\arrow["\psi"', tail reversed, from=1-3, to=3-3]
	\arrow["{\text{(\Cref{lem:ndpf-bij})}}", draw=none, from=1-3, to=3-3]
	\arrow["{\overline\rho_i}", dashed, from=3-1, to=3-3]
	\arrow["\pi"', two heads, from=3-1, to=5-1]
	\arrow["\pi"', two heads, from=3-3, to=5-3]
	\arrow["\rho_i", from=5-1, to=5-3]
\end{tikzcd}\]
    \caption{Illustration depicting how we use the bijections defined in this section to inductively establish \Cref{thm:main-bij}. The maps $\pi$ denote projection to the first coordinate.}
    \label{fig:inductive-diagram}
\end{figure}

We begin by introducing some additional definitions and notation that assist in brevity.

\begin{definition}\label{def:decset}
    Given a defective parking function $\x \in \NDDPF_{n, n, d}$, the \textit{decrement set} of $\x$, denoted $D(\x) \subseteq [n]$, is the set of indices $i\in[n]$ where $\delta(\x)_i \geq 0$ and there does not exists an integer $j < i$ with $\delta(\x)_j > 0$.
\end{definition}

\begin{example}
    The defective parking function $\x = (1, 1, 3, 6, 6, 6) \in \NDDPF_{6, 6,2}$ satisfies
$    \delta(\x)=(0,-1,0,2,1,0)$ and the corresponding decrement set is $D(\x) = \{1, 3, 4\}$.
\end{example}

The decrement set of a preference list plays a key role in establishing some needed bijections. Part of these bijections involve decrementing parking functions at certain indices - the decrement set consists of all of the indices which could have been decremented to arrive at the current parking function. A key component of our argument is showing that the decrement set of a defective parking function corresponds to the fixed point set of a non-defective parking function.

\begin{definition}
    Given $n\geq 1$ and  $d\geq 0$, the \emph{\pairset}
    $N_{n, d}$ is defined by
    \[N_{n, d} = \{(\x, i) \mid \x\in \NDDPF_{n,n,d}\mbox{ and } i \in D(\x)\}.\]
\end{definition}

\begin{example}
The pair set $N_{3,2}$ consists of six tuples $(\x,i)$ with $\x\in\NDDPF_{3,3,2}$ and $i\in D(\x)$. Table~\ref{tab:example n=m=3, d=2} provides the elements of $\NDDPF_{3,3,2}$, $\delta(\x)$, and the decrement sets $D(\x)$.
\begin{table}[ht]
    \centering

    \begin{tabular}{|c||c|c|c|c|c|}\hline
        $\x\in\NDDPF_{3,3,2}$ &$(1,4,4)$&$(2,4,4)$&$(3,3,3)$&$(3,3,4)$&$(3,4,4)$\\\hline
$\delta(\x)$			   &$(0,2,1)$&$(1,2,1)$&$(2,1,0)$&$(2,1,1)$&$(2,2,1)$	\\\hline
        $D(\x)$  			   &$\{1,2\}$ &$\{1\}$		&$\{1\}$&$\{1\}$&$\{1\}$		\\\hline
    \end{tabular}

    \caption{Elements of $\NDDPF_{3,3,2}$, $\delta(\x)$, and the decrement sets $D(\x)$.}
    \label{tab:example n=m=3, d=2}
\end{table}
\end{example}

From Remark~\ref{notation}, we recall that whenever we write $S(P)$, we are considering the restriction of the set $S$ with condition described by $P$. In this way, the set
$\NDDPF_{n,n,d}(x_n \leq k)$ denotes the set of elements $\x=(x_1,x_2,\ldots, x_n)\in \NDDPF_{n,n,d}$ with last entry less than or equal to $k$, and the set 
$N_{n,d-1}(x'_n \leq k-1)$ denotes the set of pairs $(\x', i) = ((x_1',x_2',\ldots,x_n'), i)\in N_{n,d-1}$ with $x_n'$ less than or equal to $k-1$.

We now establish some preliminary bijections.

\begin{lemma}\label{lem:dpf-bij}
    Let $k \in [n+1]$, $d > 0$, and $\x=(x_1,x_2,\ldots,x_n)\in \NDDPF_{n,n,d}(x_n \leq k)$.
    Define the  map \[\varphi: \NDDPF_{n,n,d}(x_n \leq k) \to N_{n,d-1}(x'_n \leq k-1)\]
    by 
    \[\varphi(\x )=(\x',i),\] 
    where 
    $\x'=
(x_1,x_2,\ldots,x_{i-1},x_{i}-1,x_{i+1}-1,\ldots,x_n-1)$ and $i\in[n]$ denotes the first index such that $\delta(\x)_i > 0$.
Then 
$\varphi$ is a bijection.
Moreover, 
    $D(\x) = \{j \in D(\x') \mid j \leq i\}$.
\end{lemma}
\begin{proof}
    Given $\x=(x_1,x_2,\ldots,x_n) \in \NDDPF_{n,n,d}$, let $i$ denote $\max(D(\x))$, the first index such that $\delta(\x)_i > 0$. Such an index existing follows from \Cref{prop:defect2} and the fact that $\x$ has positive defect. Then choose the map $\varphi$ with
    \[\varphi: \x \mapsto ((x_1,x_2,\ldots,x_{i-1},x_{i}-1,x_{i+1}-1,\ldots,x_n-1),i).\]
    Let $\x'=(x'_1,x'_2,\dots,x'_n)=(x_1,x_2,\ldots,x_{i-1},x_{i}-1,x_{i+1}-1,\ldots,x_n-1)$, be the first component of $\varphi(\x)$.

    \begin{itemize}[leftmargin=.2in]
        \item Well-defined: Since $i$ is the smallest index such that 
        $\delta(\x)_i>0$, this implies that $x_i-i>0$ and $x_{i-1}-(i-1)\leq 0$.
        Hence $x_i - 1 > i - 1$ and $x_{i-1} \leq i-1$. Thus $\x'$ is  still nondecreasing. Also, every index $j$ with $\delta(\x)_j > 0$ has $j \geq i$ by definition, so the $j$th entry of $\x'$ satisfies $x'_j=x_j - 1$. 
        Therefore, by \Cref{prop:defect2} the defect of $\x'$ satisfies \[\max_{1 \leq j \leq n}(\delta(\x')_j) = \max_{1 \leq j \leq n}(\delta(\x)_j) - 1 = d-1.\]
        This establishes that $\varphi$ is well-defined on the first component. 
        Also, because $\delta(\x)_i > 0$ and  $x'_i=x_i-1$, we know $\delta(\x')_i \geq 0$. If there were some $j < i$ with $\delta(\x')_j > 0$, then we would also have $\delta(\x)_j > 0$, as $x'_j=x_j$ for all $j\in[i-1]$. 
        Contradicting the definition of $i$. Hence $i \in D(\x')$ and $\varphi$ is well-defined on the second component.

        \item Injective: Suppose $\x=(x_1,x_2,\ldots,x_n), \y=(y_1,y_2,\ldots,y_n) \in \NDDPF_{n,n,d}$ and assume that $\varphi(\x)=\varphi(\y)$. Then the first coordinates are equal, i.e. $\x'=\y'$. This implies that $x'_j = y'_j$ for all $1 \leq j \leq n$. 
This together with the definition of $\varphi$ implies that
\begin{itemize}
\item if $j < i$, then $y_j'=x'_j=x_j$, and 
\item if  $j \geq i$, then $y_j'=x'_j=x_j-1$.
\end{itemize}

        This implies that $\x = \y$.
        \item Surjective: Suppose $(\x', i) \in N_{n,d-1}$. Let 
        \[\x = (x'_1, x'_2, \ldots, x'_{i-1}, x'_i + 1, x'_{i+1}+1,\ldots, x'_n+1).\]
        Then from the definition of $\varphi$, $\varphi(\x) = (\x', i)$. By proposition \Cref{prop:defect2} $d - 1 = \max_{j \in [m]}(\delta(\x')_j)$. Let $j$ denote an index where $\delta(\x')_j = d-1$.
        From the definition of $D(\x')$, $i \leq j$. Then $\delta(\x)_j = d$ and this is maximal, so $\x$ has defect $d$. 
        Finally, because $x'_n \leq k-1$, $x_n \leq k$, so $\x \in \NDDPF_{n,n,d}(x_n \leq k)$ as desired.
    \end{itemize}
        Therefore, $\varphi$ is a bijection.

To conclude we prove that $D(\x) = \{j \in D(\x') \mid j \leq i\}$ by establishing two set containments:
    \begin{itemize}[leftmargin=.2in]
        \item $D(\x) \subseteq \{j \in D(\x') \mid j \leq i\}$:  Suppose $j \in D(\x)$. Because $i = \max(D(\x))$, $j \leq i$. Hence, for all $\ell < j$ we know $x_\ell = x'_\ell$, so $\delta(\x')_\ell \leq 0$. \begin{itemize}
            \item If $j < i$, then $x'_j = x_j$, so $\delta(\x')_j = 0$. 
            \item If $j = i$, then $\delta(\x)_j > 0$, so $\delta(\x')_j \geq 0$. Then $j \in D(\x')$ and $j \leq i$. Thus $j \in \{j \in D(\x') \mid j \leq i\}$, as desired.
        \end{itemize}
        \item $\{j \in D(\x') \mid j \leq i\}\subseteq D(\x) $: Suppose $j \in \{j \in D(\x') \mid j \leq i\}$. Again, because $j \leq i$ we know that $x_\ell = x'_\ell$ for all $\ell < j$. Then $\delta(\x)_\ell \leq 0$. 
        \begin{itemize}
            \item If $j < i$, then $x_j = x'_j$. Hence, $\delta(\x')_j = 0$.
            \item If $j = i$, then $\delta(\x')_j \geq 0$, so $\delta(\x)_j > 0$.
        \end{itemize}  Thus $j \in D(\x)$, as desired.\qedhere
    \end{itemize}
\end{proof}

\begin{example}
    Let $\x=(1,1,3,5,7,6)\in\NDDPF_{6,6,2}$. Note that $\delta(\x)=(0,-1,0,1,2,0)$ and the first index such that $\delta(\x)_i>0$ is $i=4$. We then compute \begin{align*}
        \varphi(\x)&=((1,1,3,5-1,7-1,6-1),4)\\
        &=((1,1,3,4,6,5),4).
    \end{align*}
    For $\x'=(1,1,3,4,6,5)$ we have  $\delta(\x')=(0,-1,0,0,1,-1)$, meaning $\x'\in\NDDPF_{6,6,1}$ and $4\in D(\x')$.
\end{example}

\begin{example} \Cref{tab:phi_example} provides an example of the map $\varphi$ as a bijection between the sets $\NDDPF_{3,3,2}$ and $N_{3,1}$. Note that if the last entry of $\x$ is bounded by $k$, then the last entry of its image is indeed bounded by $k-1$.

\begin{table}[ht]
    \centering
\resizebox{\textwidth}{!}{
    \begin{tabular}{|c||c|c|c|c|c|}\hline
        $\x\in\NDDPF_{3,3,2}$ &$(1,4,4)$&$(2,4,4)$&$(3,3,3)$&$(3,3,4)$&$(3,4,4)$\\\hline
$\varphi(\x)\in N_{3,1}$			   &$((1,3,3),2)$&$((1,3,3),1)$&$((2,2,2),1)$&$((2,2,3),1)$&$((2,3,3),1)$	\\\hline
    \end{tabular}
}

    \caption{Elements of $\NDDPF_{3,3,2}$ and their images $\varphi(\x)\in N_{3,1}$.}
    \label{tab:phi_example}
\end{table}
\end{example}

We now introduce the following definitions in order to establish our next bijection.

\begin{definition}\label{def:fixedset}
    Given $\x \in \NDPF_{n,n}$, the \textit{fixed set} of $\x$ is
    \[F(\x) = \{i \in [n] \mid \delta(\x)_i = 0\}.\]
\end{definition}
Since $m = n$, $i\in F(\x)$ is equivalent to saying $x_{i} = i$.
\begin{example}
    The fixed set of the parking function $\x =( 1,1,2,3,5) \in \NDPF_{5,5}$ satisfies $\delta(\x)=
    (0,-1,-1,-1,0)
    $. Hence, $F(\x) = \{1, 5\}$.
\end{example}

\begin{definition}
    The \emph{\setM} is defined by
    \[M_{n} = \{(\x, i) \mid \x \in \NDPF_{n,n}\mbox{ and } i \in F(\x)\}.\]
\end{definition}

\begin{lemma}\label{lem:ndpf-bij}
    Let $k \in [n]$ and $\x=(x_1,x_2,\ldots,x_{n+1})\in \NDPF_{n+1,n+1}(x_n \leq k)$.
    Define the  map \[\psi: \NDPF_{n+1,n+1}(x_{n+1}\leq k) \to M_{n}(x_n \leq k)\]
    by 
    \[\psi(\x )=(\x',i),\] 
    where 
    $i = \max(F(\x))$ and
    \[\x' = (x_1,x_2,\ldots,x_{i-1},x_{i+1},x_{i+2},\ldots,x_{n+1}).\]
Then 
$\psi$ is a bijection.
Moreover, 
    $F(\x') = \{j \in F(\x) \mid j \leq i\}$.

\end{lemma}

\begin{proof} Fix $\x=(x_1,x_2,\ldots,x_{n+1}) \in \NDPF_{n+1,n+1}(x_{n+1} \leq k)$ and  let $i=\max(F(\x))$. Let \[\x'=(x'_1,x'_2,\dots,x'_n)=(x_1,x_2,\ldots,x_{i-1},x_{i+1},x_{i+2},\ldots,x_{n+1}),\] be the first component of $\psi(\x)$.

    \begin{itemize}[leftmargin=.2in]
        \item Well-defined: 
        Since $\x$ is a nondecreasing parking function, we know that $x_i\leq i$ for all $i\in[n+1]$.
        Moreover, since $i=\max(F(\x))$, we know that $x_\ell<\ell$ for all $i<\ell\leq n+1$. 
        Now removing $x_i$ from $\x$ to construct $\x'$ ensures that $\x'$
        satisfies the following: 
        \begin{itemize}
            \item If $1\leq j < i$, then $x'_j = x_j\leq j$, and 
            \item if $i\leq j\leq n$, then $x'_j = x_{j+1}< j+1$, so $x_j'\leq j$.
        \end{itemize}
        The above two condition ensure that $\x'\in\NDPF_{n,n}$. 
        By assumption $x_{n+1}\leq k$, so $x'_n=x_{n+1}\leq k$, which implies that $\x'\in\NDPF_{n,n}(x_n'\leq k)$, as desired. 
        Then $\psi$ is well-defined on the first component. 
        
        Because $k\leq n$, we have that $x_{n+1} < n+1$ is not a fixed point of $\x$, so $i < n+1$. As $i=\max(F(\x))$, we know $x_{i+1}$ is not a fixed point of $\x$, so $x_{i+1} < i+1$. Hence, \begin{equation}\label{eq:fp-stays}i = x_i \leq x_{i+1} < i+1,\end{equation}
        so $x'_i = x_{i+1} = i$. This implies that $i$ is in $F(\x')$, so $\psi$ is well-defined on the second component.
    
        \item Surjective: Suppose $(\x', i) \in M_n$, where $\x'=(x_1',x_2',\ldots,x_n')\in\NDPF_{n,n}$ with $x_n'\leq k$. Then let
        \[\x =(x_1,x_2,\ldots,x_{n+1})= (x'_1, x'_2, \ldots, x'_{i-1}, i, x'_{i}, x'_{i+1}, \ldots, x'_{n}).\]
        We show that $\x\in\NDPF_{n+1,n+1}(x_{n+1}\leq k)$ and $\psi(\x) = (\x', i)$. We proceed by showing that $x_j\leq j$ for all $j\in[n+1]$.
        In the case that $1\leq j < i$ we have $x_j = x'_j \leq j$. 
        When $j = i$ we have $x_i = i$.
        When $i<j\leq n+1$ we have $x_j = x'_{j-1} \leq j-1 \leq j$.
        Thus $x_j\leq j$ for all $j\in[n+1]$, so $\x$ is a parking function. Because $x'_i = i$, we have
        \[x_{i-1} < x_i = x_{i+1},\]
        so $\x$ is nondecreasing. Then $\x \in \NDPF_{n+1,n+1}$. Because $x_j = x'_{j-1} \leq j-1 < j$, for $j>i$, $i$ must be the index of the last fixed point of $\x$, so $\psi(\x) = (\x', i)$ as desired.

        \item Injective: 
         Suppose $\x=(x_1,x_2,\ldots,x_{n+1}), \y=(y_1,y_2,\ldots,y_{n+1}) \in \NDPF_{n+1,n+1}$ and assume that $\psi(\x)=\psi(\y)$. 
         This implies that  $\max(F(\x))=\max(F(\y))$, which we denote by $i$. 
         Moreover, the first coordinates are equal, i.e. $\x'=\y'$, which implies that $x'_j = y'_j$ for all $j\in[n+1]\setminus\{i\}$. By the definition of $\psi$, this implies that $x_j = y_j$ for all $j\in[n+1]\setminus\{i\}$. Then as $i=\max(F(\x))=\max(F(\y))$, we have that $x_i=i=y_i$. Thus $\x=\y$, and $\psi$ is injective.
    \end{itemize}
    Hence, $\psi$ is a bijection. 
    
    To conclude, we prove that $F(\x) = \{j \in F(\x') \mid j \leq i\}$ by establishing two set containments:
    \begin{itemize}[leftmargin=.2in]
        \item $F(\x) \subseteq \{j \in F(\x') \mid j \leq i\}$: Suppose $j \in F(\x)$. Recall $i=\max(F(\x))$. Hence,
        \[x'_j = \begin{cases}
            x_j & \mbox{if }1\leq j < i\\
            x_{j+1} & \mbox{if }j=i.
        \end{cases}\]
        If $1\leq j < i$, then  $x'_j =x_j < j$. 
        If $j = i$, then from equation \eqref{eq:fp-stays} we know that $x_{j+1} = j$, so $j \in F(\x')$. Thus $j\in\{F(\x')\mid j\leq i\}$.
        \item $ \{j \in F(\x') \mid j \leq i\}\subseteq F(\x)$: Suppose $j \in F(\x')$ and $1\leq j \leq i$. Then we have
        \[x_j = \begin{cases}
            x'_j &\mbox{if } 1\leq j < i\\
            i & \mbox{if }j=i.
        \end{cases}\]
        Hence, if $1\leq j< i$, then $\delta(\x)_j=x_j-j=x'_j-j=0$ and $j\in F(\x)$. If $j=i$, then $\delta(\x)_j=x_j-j=i-i=0$ and so $j\in F(\x)$. 
        In both cases we get that         $j \in F(\x)$, as desired.\qedhere
    \end{itemize}
\end{proof}

\begin{example} \Cref{tab:psi_example} provides an example of the map $\psi$ as a bijection between the sets $\NDPF_{4,4}(x_4\leq 3)$ and $M_3(x_4\leq 3)$.

\begin{table}[ht]
    \centering

    \begin{tabular}{|c|c|}\hline
        $\x\in\NDPF_{4,4}(x_4\leq 3)$ & $\psi(\x)\in M_3(x_4\leq 3)$\\ \hline\hline
        (1,1,1,1) & ((1,1,1),1)\\ \hline
        (1,1,1,2) & ((1,1,2),1)\\ \hline
        (1,1,1,3) & ((1,1,3),1)\\ \hline
        (1,1,2,2) & ((1,2,2),1)\\ \hline
        (1,1,2,3) & ((1,2,3),1)\\ \hline
        (1,1,3,3) & ((1,1,3),3)\\ \hline
        (1,2,2,2) & ((1,2,2),2)\\ \hline
        (1,2,2,3) & ((1,2,3),2)\\ \hline
        (1,2,3,3) & ((1,2,3),3)\\ \hline
    \end{tabular}

    \caption{Elements of $\NDPF_{4,4}(x_4\leq 3)$ and their images $\psi(\x)\in M_3(x_4\leq 3)$.}
    \label{tab:psi_example}
\end{table}
\end{example}

\begin{theorem}\label{thm:main-bij}
    For fixed $k \in [n+1]$, there is a bijection
    \[\rho: \NDDPF_{n,n,d}(x_n \leq k) \to \NDPF_{n+d,n+d}(x_{n+d} \leq k-d)\]
    such that $\x \in \NDDPF_{n,n,d}(x_n\leq k)$ satisfies $D(\x) = F(\rho(\x))$.
\end{theorem}

\begin{proof}
    We prove that there exist bijections \[\rho_r: \NDDPF_{n,n,r}(x_n \leq k-d+r) \to \NDPF_{n+r,n+r}(x_{n+r} \leq k-d)\] for $r\geq0$ such that $D(\x)=F(\rho_r(\x))$, and the desired bijection is obtained as the case when $r=d$.

    We proceed by induction on $r\geq0$. For the base case $r=0$, note that $\NDDPF_{n,n,0} = \NDPF_{n,n}$. The elements of the decrement set of a parking function $\x$ with no defect are precisely the elements of the fixed point set of $\x$, so we take $\rho_0$ to be the identity map.

    Now for the purpose of induction, suppose there is a bijection \[\rho_r: \NDDPF_{n,n,r}(x_n \leq k-d+r) \to \NDPF_{n+r,n+r}(x_{n+r} \leq k-d)\] such that $D(\x)=F(\rho_r(\x))$. Because $D(\x)=F(\rho_r(\x))$, the bijection $\rho_r$ determines a bijection \[\overline\rho_r:N_{n,r}(x_n\leq k-d+r)\to M_{n+r}(x_{n+r}\leq k-d)\] where $\overline\rho_r(\x,i)=(\rho_r(\x),i)$. From Lemmas~\ref{lem:dpf-bij} and~\ref{lem:ndpf-bij}, we have bijections
    \begin{align*}
        \varphi&: \NDDPF_{n,n,r+1}(x_n \leq k-d+r+1) \to N_{n,r}(x_n \leq k-d+r),\mbox{  and}\\
        \psi&: \NDPF_{n+r+1,n+r+1}(x_{n+r+1} \leq k-d) \to M_{n+r}(x_{n+r} \leq k-d)
    \end{align*}
    that map
    \begin{align*}
        \varphi&: \x \mapsto (\x', i), \mbox{ and }\\
        \psi&: \y \mapsto (\y', i)
    \end{align*}
    with
    \begin{align*}
        D(\x) &= \{j \in D(\x') \mid j \leq i\},\mbox{ and }\\
        F(\y) &= \{j \in F(\y') \mid j \leq i\}.
    \end{align*}

    We then have a bijection \[\rho_{r+1}: \NDDPF_{n,n,r+1}(x_n \leq k-d+r+1) \to \NDPF_{n+r+1,n+r+1}(x_{n+r+1} \leq k-d)\] where $\rho_{r+1}=\psi^{-1}\circ\overline\rho_r\circ\varphi$. It remains only to show that $D(\x)=F(\rho_{r+1}(\x))$. Indeed,
    \begin{align*}
        F(\rho_{r+1}(\x))&=F(\psi^{-1}\circ \overline\rho_r\circ\varphi(\x))\\
        &=F(\varphi^{-1}\circ\overline\rho_r(\x',i))\\
        &=F(\psi^{-1}(\rho_r(\x'),i))\\
        &=F(\y)\\
        \intertext{ where $\psi(\y)=(\rho_r(\x'),i)$. So,}
        F(\rho_{r+1}(\x))&=F(\y)\\
        &=\{j\in F(\rho_r(\x'))\mid j\leq i\}\\
        &=\{j\in D(\x')\mid j\leq i\}\\
        &=D(\x).\qedhere
    \end{align*}
\end{proof}

\begin{example}
    Although the bijection $\rho$ given in \Cref{thm:main-bij} is defined inductively, one can compute $\rho$ of a defective parking function in practice. Let $\x=(1,1,3,6,6,7)\in\NDDPF_{6,6,2}$. In the diagram below, we apply $\varphi$ twice to obtain defective parking functions with smaller defect, recording the value $i$ along the arrow. We then apply $\psi^{-1}$ with the same $i$ values in reverse order to obtain $\rho(\x)=(1,1,3,4,4,5,5,5)$.

\[
\resizebox{\textwidth}{!}{
\begin{tikzcd}[ampersand replacement=\&]
	{(1,1,3,5,7,7)\in\NDDPF_{6,6,2}(x_6\leq7)} \&\& {\delta=(0,-1,0,1,2,1)} \\
	{(1,1,3,4,6,6)\in\NDDPF_{6,6,1}(x_6\leq6)} \&\& {\delta=(0,-1,0,0,1,0)} \\
	{(1,1,3,4,5,5)\in\NDDPF_{6,6,0}(x_6\leq5)} \&\& {\delta=(0,-1,0,0,0,-1)} \\
	{(1,1,3,4,5,5,5)\in\NDPF_{7,7}(x_6\leq5)} \&\& {\delta=(0,-1,0,0,0,-1,-2)} \\
	{(1,1,3,4,4,5,5,5)\in\NDPF_{8,8}(x_6\leq5)} \&\& {\delta=(0,-1,0,0,-1,-1,-2,-3)}
	\arrow["{(i=4)}", from=1-1, to=2-1]
	\arrow["\varphi"', from=1-1, to=2-1]
	\arrow["{(i=5)}", from=2-1, to=3-1]
	\arrow["\varphi"', from=2-1, to=3-1]
	\arrow["{(i=5)}", from=3-1, to=4-1]
	\arrow["{\psi^{-1}}"', from=3-1, to=4-1]
	\arrow["{(i=4)}", from=4-1, to=5-1]
	\arrow["{\psi^{-1}}"', from=4-1, to=5-1]
\end{tikzcd}
}
\]

\end{example}

Setting $k=n+1$ in Theorem~\ref{thm:main-bij} implies the following result.  

\begin{corollary}\label{cor:n=m}
    For any $n\in\NN$ and nonnegative integer $d$, there is a bijection between the sets $\NDDPF_{n,n,d}$ and $\NDPF_{n+d,n+d}(x_{n+d} \leq n+1-d)$. 
\end{corollary}

\subsection{Different Numbers of Cars and Spots}
We now consider the more general case where there are $m\in\NN$ cars and  $n\in\NN$ parking spots.
We begin by establishing an analogous result to Corollary~\ref{cor:n=m} in the case when $m>n$.
\begin{proposition}\label{prop:m>n}
    For $m > n$, there exists a bijection 
    \[\theta:\NDDPF_{m,n,d} \bij \NDPF_{n+d,n+d}(x_{n+d} \leq m+1-d).\]
\end{proposition}
\begin{proof}
    Consider $\x \in \NDDPF_{m,n,d}$. Because $x_1 \geq 1$, we have that 
    \[d = \max_{i \in [m]}(\delta(\x)_i) \geq \delta(\x)_1 \geq m-n.\]
    Introducing $k$ new spots at the end of the street  reduces the defect of $\x$ by $k$. So in particular, by adding $m-n$ spots we have a bijection
    \[\theta_1:\NDDPF_{m,n,d} \bij \NDDPF_{m,m,d-(m-n)}(x_{m} \leq n+1).\]
    Theorem~\ref{thm:main-bij} yields a bijection
    \[\rho:\NDDPF_{m,m,d-(m-n)}(x_m \leq n+1) \bij \NDPF_{n+d,n+d}(x_{n+d} \leq m+1-d).\]
    The desired result comes from the composition $\theta=\rho \circ\theta_1$.
\end{proof}

\begin{definition}\label{def:conjugate-lat}
The \textit{conjugate lattice path} of $\w$, which we denote by $\bar{\w}$, is obtained by reversing $\w$ and then switching $E$'s to $N$'s and $N$'s to $E$'s. The \emph{conjugate of a nondecreasing  preference list} $\x$, denoted by $\bar{\x}$, is obtained by conjugating its corresponding lattice path, then turning the result back into a parking function.
\end{definition}

The key intuition in Definition~\ref{def:conjugate-lat} is that conjugating the lattice path corresponds to ``flipping'' the path along the line $y=-x$ in a process similar to conjugating a Young diagram.

\begin{example} Consider $\x=(1,1,2,3,5,5,6)$. The $(5,7)$ lattice path corresponding to $\x$ is given by $\w=NNENENEENNEN$. Then the corresponding conjugate path to $\w$ is 
\[\bar{\w}=ENEENNENENEE.\] 
Both  $\w$ and $\bar{\w}$ are illustrated in Figure~\ref{fig:path}. We conclude by noting that based on $\bar{\w}$, one can readily confirm that $\bar{\x}=(2,4,4,5,6)$.

\begin{figure}[ht] 
   \centering
    \resizebox{1.5in}{!}{
\usetikzlibrary{decorations.markings}
\begin{tikzpicture}
\draw[gray](0,0) grid (5,7);
\draw[red, ultra thick](0,0)--(0,2)--(1,2)--(1,3)--(2,3)--(2,4)--(4,4)--(4,6)--(5,6)--(5,7);
\draw[blue, thick, dashed] (0,2)--(5,7) node [above] {$y=x+2$};

\node at (2.5,-.75) {$\w=NNENENEENNEN$};
\end{tikzpicture}}
\qquad\qquad
\resizebox{!}{1.5in}{
\begin{tikzpicture}
\draw[gray] (0,0) grid (7,5);
\draw[red, ultra thick](0,0)--(1,0)--(1,1)--(3,1)--(3,3)--(4,3)--(4,4)--(5,4)--(5,5)--(7,5);
\draw[blue, thick, dashed] (0,0)--(5,5) node [above] {$y=x$};
\draw[black, thick, dotted] (2,0)--(7,5) node [above] {$y=x-2$};
\node at (3.5,-.75) {$\bar{\w}=ENEENNENENEE$};
\end{tikzpicture}}
    \caption{The lattice path $\w$ corresponding to $\x=(1,1,2,3,5,5,6)$ as well as its conjugate lattice path $\bar{\w}$ corresponding to $\bar{\x} = (2,4,4,5,6)$. This figure also illustrates how the line $y=x+2$ is taken to the line $y=x$ by conjugation.}
    \label{fig:path}
\end{figure}
\begin{figure}[h!]
\usetikzlibrary{decorations.markings}
\end{figure}
\end{example}

\begin{proposition}\label{prop:conj}
    If $\x\in\NDDPF_{m,n,d}$, then $\bar\x\in\NDDPF_{n,m,d+(n-m)}$. Furthermore, if \[c:\NDDPF_{m,n,d} \to \NDDPF_{n,m,d+(n-m)}\] is defined by $c(\x)=\bar{\x}$, then $c$ is a bijection.
\end{proposition}

\begin{proof}
    Let $\x\in\NDDPF_{m,n,d}$. By \Cref{thm:defect_characterization}, $\x$ corresponds via Equation~\ref{eq:pref_path_bij} to an $(n,m)$-lattice path $\w$ whose maximum vertical distance below the line $y=x+(m-n)$ is $d$. Because the line $y=x+(m-n)$ passes through the point $(n,m)$, we have that conjugation takes the line $y=x+(m-n)$ to the line with slope one passing through $(0,0)$, that is, $y=x$ (see \Cref{fig:path} for an example). A point on the path $\w$ that is $k$ below $y=x+(m-n)$ then becomes a point on the path $\bar\w$ that is $k$ below $y=x$. This point is then $k+(n-m)$ below the line $y=x+(n-m)$. The largest such distance is achieved when $k=d$, hence the preference list $\bar\x$ is contained in $\NDDPF_{n,m,d+(n-m)}$.

    Because conjugation of lattice paths is an involution, we have immediately that the map $c:\NDDPF_{m,n,d} \to \NDDPF_{n,m,d+(n-m)}$ is a bijection.
\end{proof}

\begin{proposition}\label{prop:m<n}
    If $m < n$, there is a bijection
    \[\gamma:\NDDPF_{m,n,d} \bij \NDPF_{n+d,n+d}(x_{n+d} \leq m+1-d).\]
\end{proposition}
\begin{proof}
    Proposition~\ref{prop:conj} gives a bijection
    \[\gamma_1:\NDDPF_{m,n,d} \bij \NDDPF_{n,m,d-(m-n)}.\]
    Proposition~\ref{prop:m>n} gives a bijection
    \[\theta:\NDDPF_{n,m,d-(m-n)} \bij \NDPF_{n+d,n+d}(x_{n+d} \leq m+1 - d).\]
    The desired result comes from the composition $\gamma=\theta\circ \gamma_1$.
\end{proof}

\begin{corollary}\label{cor:rational to classical}
For all $m,n\in\NN$ and integer $d$ satisfying $n-m\leq d\leq n$, the sets $\NDDPF_{m,n,d}$ and $\NDPF_{n+d,n+d}(x_{n+d} \leq m+1-d)$ are in bijection.
\end{corollary}

\begin{proof}
    Either $m=n$, $m>n$, or $m<n$. These cases are covered by Corollary~\ref{cor:n=m}, Proposition~\ref{prop:m>n}, and Proposition~\ref{prop:m<n}, respectively.
\end{proof}

Next we establish a bijection between nondecreasing parking functions with restricted maximal values and two-row standard Young tableaux. This connection is widely accepted as fact, however after a literature search we were unable to find the original proof. Thus, for sake of completion, we include an independent proof below.

If $m=n$ and an $(n,n)$-lattice path lies above the line $y=x$, then the lattice path is called a Dyck path of semilength $n$. 
We let $\mathcal{D}_n$ denote the set of Dyck paths of semilength $n$.

\begin{proposition}\label{prop:SYT}
For any positive integers $a\geq b$, let $\SYT(a, b)$ denote the set of standard Young tableaux of shape $(a,b)$. For positive integers $n, k$ with $k \leq n$, there is a bijection
    \[\sigma: \NDPF_{n,n}(x_n \leq k) \to \SYT(n, k - 1).\]
\end{proposition}
\begin{proof}
    Let $\z$ be an element of $\mathcal{D}_n$. There is a bijection $\sigma_1: \NDPF_{n,n} \to \mathcal{D}_n$ such that if $\x \in \NDPF_{n,n}$, the $i$th North step in $\sigma_1(\x)$ is preceded by exactly $x_i - 1$ East steps. East steps are appended to the end of the Dyck path to ensure proper length. This is well defined because $x_i \leq i$, so the $i$th East step always comes after the $i$th North step. It's injective because two parking functions that map to the same Dyck path must be equal in every index, and it's surjective because reversing the process yields a valid parking function.
    
A bijection $\sigma_2:\mathcal{D}_n\to \SYT(n,n)$ is given in \cite[Page 1]{DyckPathsSYT} where 
for a Dyck path $\z$ of semilength $n$, they define the tableau $\sigma_2(\z)$ by putting $i$ in the first row if $z_i = N$ and in the second row if $z_i = E$.
    
    Restricting $\sigma_1$ to $\NDPF_{n,n}(x_n \leq k)$ yields a bijection to the set of Dyck paths of semilength $n$ that end with $n-k+1$ East steps. Restricting $\sigma_2$ yields a bijection to the set of Young tableaux of shape $(n,n)$ such that the numbers $\{n+k+1, \ldots, 2n\}$ all lie in the second row. Then because we know that the last $n-k+1$ entries of the second row are always the same, we can remove those entries to obtain a bijection.
\end{proof}

We are now ready to prove Theorem~\ref{thm:general}, which we recall states that the set of nondecreasing defective parking functions of defect $d$, denoted $\DPF_{m,n,d}^{\uparrow},$ and the set of standard Young tableaux of shape $(n+d,m-d)$ are in bijection. 
\begin{proof}[Proof of Theorem~\ref{thm:general}]
The result follows from the composition of the bijections in Corollary~\ref{cor:rational to classical} and Proposition~\ref{prop:SYT}.
\end{proof}

\section{Enumerative Results}\label{sec:enumerative}

We are now ready to establish a count for the number of $\Sym_n$-orbits of defective parking functions. Given the bijection in Theorem~\ref{thm:general}, it suffices to use  the hook length formula to establish the count of Corollary~\ref{cor:counting}.

\begin{proof}[Proof of Corollary~\ref{cor:counting}]
    From Corollary~\ref{cor:rational to classical}, there is a bijection \[\NDDPF_{m,n,d} \bij \NDPF_{n+d,n+d}(x_{n+d} \leq m+1-d).\]
    By Proposition~\ref{prop:SYT}, $\NDPF_{n+d,n+d}(x_{n+d} \leq m+1-d)$ is in bijection with standard Young tableaux of shape $(n+d,m-d)$. The hook length formula implies
    \[|\NDDPF_{m,n,d}| = \frac{n-m+2d+1}{n+d+1}\binom{m+n}{n+d}.\qedhere\]
\end{proof}

We conclude by providing a formula for the size of each orbit. 

\begin{theorem}\label{thm:size of orbit}
Given a nondecreasing defective parking function $\x \in \NDDPF_{m,n,d}$, let $\ell_i$ denote the number of elements of $\x$ equal to $i$. Then the size of the orbit of defective parking functions with representative $\x$ is
\[
\binom{m}{\ell_1,\ell_2,\ldots,\ell_{n}}=\frac{m!}{\ell_1!\ell_2!\cdots \ell_{n}!}.\] 
\end{theorem}
\begin{proof}
The size of the orbit is equal to the number of multiset permutations of $\x$, which is given by the multinomial coefficient.
\end{proof}

Using our previous results, we establish Theorem~\ref{thm:new formula}, thereby giving a new formula for the number of defective parking functions. 
\begin{proof}[Proof of Theorem~\ref{thm:new formula}]
    The formula in equation \eqref{eq:full count}  follows directly from summing Theorem~\ref{thm:size of orbit} over $\NDDPF_{m,n,d}$.
\end{proof}

\section{Future Work}
We conclude with the following directions for future research. 

\begin{enumerate}
\item We welcome a proof of the conjectured formula for the defective Kreweras numbers as stated in Conjecture~\ref{conj:defective Kreweras}. If a counterexample is found, we then ask for a correct formula for these numbers.

\item Our proof of Theorem~\ref{thm:general} relied on a chain of bijective arguments. We wonder if a more direct bijection can be established. 

\item We studied one expansion of the Frobenius characteristic by introducing the defective Kreweras numbers,
but one could try to expand this symmetric function in other classical bases.
Of particular interest is the expansion in the power-sum basis, which records the number of defective parking functions fixed by a permutation of a particular cycle type.
The first step in this direction is likely a simpler formula than Theorem~\ref{thm:new formula}, which solves this question for the case of the identity permutation.

\item 
There have been numerous works generalizing parking functions and studying discrete statistics on these sets. We pose the following problems.

\begin{enumerate}
    \item Consider cars of varying lengths (such as parking assortments and sequences \cites{assortments1,sequences}). Give counts for the number of defective parking assortments and sequences.
    \item Consider the case where cars have an interval of parking preferences, as in \cite{ell_interval_pfs}). Give counts for the number of defective $\ell$-interval parking functions.
    \item Enumerate defective parking functions based on the number of lucky cars, those cars which park in their preference. When the defect is $d=0$, and $m=n$, there are known product formulas for their enumeration \cite{GesselSeo}. 
\end{enumerate}
\end{enumerate}

\section*{Acknowledgments}
The authors acknowledge funding support from National Science Foundation Grant No.\ DMS-1659138 and Sloan Grant No.\ G-2020-12592. P.~E.~Harris was partially supported through a Karen Uhlenbeck EDGE Fellowship.

\bibliographystyle{acm}
\bibliography{bibliography}

\begin{Contacts}
\AuthorAddress%
  {Rebecca E. Garcia}
  {Department of Mathematics and Computer Science, Colorado Springs, CO 80903 United States}
  {rgarcia2023@coloradocollege.edu}
\AuthorAddress%
  {Pamela E. Harris}
  {Department of Mathematical Sciences, University of Wisconsin-Milwaukee, Milwaukee, WI 53211 United States}
  {peharris@uwm.edu}
\AuthorAddress%
  {Alex Moon}
  {Department of Mathematical Sciences, University of Wisconsin-Milwaukee, Milwaukee, WI 53211 United States}
  {ajmoon@uwm.edu}
\AuthorAddress%
  {Aaron Ortiz}
  {Department of Mathematics, The University of Kansas, Lawrence, KS 66045 United States}
  {aortiz@ku.edu}
\AuthorAddress%
  {Lauren J. Quesada}
  {Department of Statistics, Colorado State University, Fort Collins, CO 80523-1877 United States}
  {Lauren.Quesada@colostate.edu}
\AuthorAddress%
  {Cynthia Marie Rivera S\'anchez}
  {Department of Mathematics, University of Puerto Rico - R\'io Piedras, San Juan PR, 00925-2537 United States}
  {cynthia.rivera15@upr.edu}
\AuthorAddress%
  {Dwight Anderson Williams II}
  {Department of Mathematics, Morgan State University, Baltimore, MD 21251 United States}
  {dwight@mathdwight.com}
\AuthorAddress%
  {Alexander N.~Wilson}
  {Department of Mathematics, Oberlin College, Oberlin, OH 44074}
  {math@alexandernwilson.com}
\end{Contacts}
\end{document}